\documentclass[12pt,bezier]{article}
\usepackage{amsmath, amssymb, amsfonts, euscript, pifont, mathrsfs, latexsym, graphicx, mathcomp, bbm, dsfont, yfonts, bbding}

\newtheorem{prethm}{{\bf Theorem}}

\newenvironment{thm}{\begin{prethm}\sl{\hspace{-0.5
               em}{\bf.}}}{\end{prethm}}

\newtheorem{prepro}[prethm]{{\bf Proposition}}

\newtheorem{prelem}[prethm]{{\bf Lemma}}

\newenvironment{lem}{\begin{prelem}\sl{\hspace{-0.5
               em}{\bf.}}}{\end{prelem}}

\newtheorem{precor}[prethm]{{\bf Corollary}}

\newenvironment{cor}{\begin{precor}\sl{\hspace{-0.5
               em}{\bf.}}}{\end{precor}}

\newtheorem{preconj}[prethm]{{\bf Conjecture}}

\newtheorem{preremark}[prethm]{{\bf Remark}}

\newtheorem{preexample}[prethm]{{\bf Example}}

\newtheorem{prethmm}{{\bf Theorem}}

\newtheorem{preproof}{{\bf\textsf{Proof.}}}

\newenvironment{proof}[1]{\begin{preproof}{\rm
               #1}\hfill{$\square$}}{\end{preproof}}

\newcommand{\bmi}[1]{\mbox{\boldmath $ #1$}}
\def\mathbi#1{\textbf{\em#1}}
\DeclareMathAlphabet{\mathpzc}{OT1}{pzc}{m}{it}

\title{\bf\Large  On a family of diamond-free strongly regular graphs}

\author{{\sc\normalsize  A. Mohammadian \quad\quad\quad B. Tayfeh-Rezaie}\vspace{2mm}
\\{\footnotesize{\sl School of Mathematics, Institute for Research in Fundamental Sciences$\,$(IPM),}}\vspace{-1.5mm}
\\{\footnotesize{\sl  P.O. Box 19395-5746, Tehran, Iran}}\vspace{2mm}
\\{\footnotesize{$\mathsf{ali\_m@ipm.ir}$ \quad\quad\quad  $\mathsf{tayfeh}$-$\mathsf{r@ipm.ir}$}\vspace{1cm}}}

\date{}

\begin{document}
\maketitle

\begin{abstract}
{\small
The existence of  a partial quadrangle  ${\mathsf{PQ}}(s, t, \mu)$ is equivalent  to the  existence of a  diamond-free  strongly regular graph ${\mathsf{SRG}}(1+s(t+1)+s^2t(t+1)/\mu, s(t+1), s-1, \mu)$.  Recently,  it is shown   that there exists a  ${\mathsf{PQ}}(2, (n^3+3n^2-2)/2, n^2+n)$
if and only if   $n\in\{1, 2, 4\}$. Let $\mathcal{S}$  be a  ${\mathsf{PQ}}(3,(n+3)(n^2-1)/3, n^2+n)$ such that for every  two non-collinear points $p_1$ and $p_2$, there is a point $q$ non-collinear with  $p_1$, $p_2$,  and all  points collinear with both $p_1$ and $p_2$. In this article, we establish  that  $\mathcal{S}$ exists only for  $n\in\{-2, 2, 3\}$ and probably $n=10$.
}

\vspace{5mm}

\noindent{\small \mathbi{Key words and phrases}:    adjacency matrix,  eigenvalue multiplicity,  automorphism group,  diamond-free graph,   negative Latin square graph, partial quadrangle, strongly regular graph, transitive graph.}

\vspace{3mm}

\noindent{\small \mathbi{AMS  Mathematics Subject Classif{}ication\,(2010)}: 05B25, 05C25, 05C50, 05E30.}
\end{abstract}

\vspace{5mm}

\section*{I. Introduction}

A {\sl strongly regular graph} with parameters $(\nu, k, \lambda, \mu)$, denoted by  ${\mathsf{SRG}}(\nu, k, \lambda, \mu)$,  is a regular graph of order  $\nu$ and valency  $k$  such that  $\mathsf{(i)}$ it is not complete or edgeless,  $\mathsf{(ii)}$   every two adjacent vertices have $\lambda$ common neighbors,  and $\mathsf{(iii)}$ every two non-adjacent vertices have $\mu$ common neighbors.
The concept of strongly regular graphs was f{}irst introduced
by  Bose and  Shimamoto in  \cite{bos}.
Strongly regular graphs form an important class of graphs and  lie
somewhere between highly structured graphs  and  apparently random graphs.
They  often appear in dif{}ferent areas such as  coding theory, design theory,
discrete  geometry,  group theory, and so on.
Obviously,  complete multipartite graphs with  equal part sizes  and their complements   are trivial examples of strongly regular graphs. To exclude these examples, we assume that a strongly regular graph and its complement are connected; or  equivalently, $0<\mu<k<\nu-1$.

The {\sl adjacency matrix}  of a graph    $G$, denoted by  $\EuScript{A}_G$,  has  its   rows and
columns  indexed by  the vertex set of $G$   and its     $(i, j)$-entry  is $1$ if the vertices $i$ and
$j$ are adjacent and $0$ otherwise.
The zeros of the characteristic polynomial of $\EuScript{A}_G$
are called the {\sl eigenvalues} of $G$.
The statement that $G$ is an ${\mathsf{SRG}}(\nu, k, \lambda, \mu)$ is
equivalent to
$$\EuScript{A}_GJ_\nu=kJ_\nu \quad \text{and} \quad  \EuScript{A}_G^2 + (\mu-\lambda)\EuScript{A}_G+ (\mu - k)I_\nu=\mu J_\nu,$$
where $I_t$ and $J_t$ are  the $t\times t$ identity matrix and the $t\times t$ all one matrix, respectively.
It is easy to verify    that the  eigenvalues of an  ${\mathsf{SRG}}(\nu, k, \lambda, \mu)$ are
\begin{align*}
&k,   \textrm{ with the multiplicity } 1; \\
&r=\frac{\lambda-\mu+\mathnormal{\Delta}}{2},   \textrm{ with the multiplicity  } f=\frac{\nu-1}{2}-\frac{2k+(\nu-1)(\lambda-\mu)}{2\mathnormal{\Delta}}; \\
&s=\frac{\lambda-\mu-\mathnormal{\Delta}}{2},  \textrm{ with the  multiplicity  } g=\frac{\nu-1}{2}+\frac{2k+(\nu-1)(\lambda-\mu)}{2\mathnormal{\Delta}},
\end{align*} where $\mathnormal{\Delta}=\sqrt{(\lambda-\mu)^2 + 4(k-\mu)}$. It is well known that the second largest eigenvalue of a graph $G$ is non-positive if and only if the non-isolated vertices of $G$ form a complete multipartite graph.  Also, it is a known fact that  the smallest eigenvalue of a graph $G$ is at least  $-1$ if and only if $G$ is a  disjoint union of some complete graphs.  So, for any ${\mathsf{SRG}}(\nu, k, \lambda, \mu)$, we necessarily have $r>0$ and $s<-1$.

The {\sl diamond} is the graph on four vertices with f{}ive edges. A graph with  no  diamond as an induced subgraph is called {\sl diamond-free}. It is straightforward  to see   that a graph is diamond-free if and only if the neighborhood of any vertex is a  disjoint union of some complete graphs.  Furthermore, an ${\mathsf{SRG}}(\nu, k, \lambda, \mu)$ is diamond-free if and only if $\lambda+1\, |\, k$ and  the neighborhood of each  vertex is $\tfrac{k}{\lambda+1}K_{\lambda+1}$.

A {\sl partial quadrangle} with parameters $(s, t, \mu)$, denoted by  ${\mathsf{PQ}}(s, t, \mu)$,  is an incidence structure $(\mathcal{P}, \mathcal{L}, \mathcal{I})$
in which  $\mathcal{P}$ and $\mathcal{L}$ are disjoint
non-empty sets of elements called points and lines, respectively, and  $\mathcal{I}\subseteq(\mathcal{P}\times\mathcal{L})\cup(\mathcal{L}\times\mathcal{P})$ is a symmetric    incidence relation satisfying the following conditions:
\vspace{-1mm}
\begin{itemize}
\item[$\mathsf{(i)}$] Each line is incident with $s+1$ points and each point is incident with $t+1$ lines.
\vspace{-2mm}
\item[$\mathsf{(ii)}$] Every two distinct points are incident with at most one line.
\vspace{-2mm}
\item[$\mathsf{(iii)}$] For  each  non-incident pair   $(p, \ell)\in\mathcal{P}\times\mathcal{L}$, there is  at most one  pair $(p', \ell')\in\mathcal{P}\times\mathcal{L}$ such that the both     $p, p'$ are  incident with $\ell'$ and  $p'$ is incident with $\ell$.
    \vspace{-2mm}
\item[$\mathsf{(iv)}$] For every  two non-collinear points, there are exactly $\mu$ points collinear with both  of them.
\end{itemize}
\vspace{-1mm}
Partial quadrangles were f{}irstly  introduced  by  Cameron in \cite{cam}.
Clearly,  for  any  ${\mathsf{PQ}}(s, t, \mu)$,  we necessarily have $\mu\leqslant t+1$. In the literature, a  ${\mathsf{PQ}}(s, t, t+1)$ is called a {\sl generalized quadrangle} and is denoted by  ${\mathsf{GQ}}(s, t)$.
The {\sl collinearity graph}
of a ${\mathsf{PQ}}(s, t, \mu)$ is the graph whose vertices are the
points and two vertices are adjacent if they are collinear.
It is straightforward to verify   that the collinearity graph
of a ${\mathsf{PQ}}(s, t, \mu)$  is a diamond-free $${\mathsf{SRG}}\left(1+s(t+1)+\frac{s^2t(t+1)}{\mu}, s(t+1), s-1, \mu\right).$$
Inversely,  a diamond-free strongly regular graph is the collinearity graph   of a  partial quadrangle
whose points are vertices of the graph and lines are  maximal cliques of the graph. So, an ${\mathsf{SRG}}(\nu, k, \lambda, \mu)$  with  $\lambda\leqslant1$ or $\mu=1$ is the collinearity graph of a partial quadrangle.

Recently,   Bondarenko and   Radchenko  showed  in \cite{bon}  that  a ${\mathsf{PQ}}(2, (n^3+3n^2-2)/2, n^2+n)$,    or equivalently, an ${\mathsf{SRG}}((n^2+3n-1)^2,n^2(n+3),1,n(n+1))$,  exists if and only if   $n\in\{1,2,4\}$. Let $\mathcal{S}$  be a  ${\mathsf{PQ}}(3,(n+3)(n^2-1)/3, n^2+n)$ such that for every  two non-collinear points $p_1$ and $p_2$, there is a point $q$ non-collinear with  $p_1$, $p_2$,  and all  points collinear with both $p_1$ and $p_2$. In this article, we will show   that if  $\mathcal{S}$ exists, then  $n\in\{-2, 2, 3, 10\}$. Equivalently, we will establish  the following theorem.
\begin{thm}\label{main}
If there exists a diamond-free ${\mathsf{SRG}}((n^2+3n-2)^2, n(n^2+3n-1), 2, n(n+1))$, for some integer  $n$,  satisfying the following   condition:
\begin{equation}\label{con}
\begin{array}{cc}
\text{For every two non-adjacent vertices $u$ and  $v$, there is  a  vertex that} \\
\text{\hspace{-11mm} is not  adjacent to   $u$, $v$, and all common neighbors of $u$ and $v$,}
\end{array}
\end{equation}
then    $n\in\{-2, 2, 3, 10\}$.
\end{thm}
\vspace{-1mm}
In each of two cases $n=-2$ and $n=2$, there is a unique diamond-free strongly regular graph  \cite{hae}. For  $n=3$, we are aware of  only one  diamond-free strongly regular graph which is found  in \cite{van}. Note   that all these three examples satisfy (\ref{con}). The question whether there exists  a diamond-free strongly regular graph for $n=10$ is left as an open problem. Finally,
we  believe that  Theorem \ref{main} holds without assuming  the condition  (\ref{con}).

\section*{II. Notation and Preliminaries}

We f{}irst  recall  some  notation from  graph  theory. For a graph $G$, the vertex set of $G$ is denoted by $V(G)$. We employ  the notation $u\thicksim v$ when  two vertices $u, v\in V(G)$ are adjacent.  For any  vertices $v_1, \ldots, v_t\in V(G)$, we let  $$N(v_1, \ldots, v_t)=\{x\in V(G)\, | \, x\thicksim v_i, \text{ for } i=1, \ldots, t\}.$$ For every two subsets $S$ and $T$ of    $V(G)$, we denote   by $\langle S, T\rangle$ the induced subgraph of $G$ on all edges with one   endpoint in $S$ and the other endpoint  in $T$. For simplicity, we  will use the notation   $N[v]$,  $\overline{N}(v)$, and $\langle S\rangle$ instead of  $N(v)\cup\{v\}$, $V(G)\setminus(N(v)\cup\{v\})$,    and $\langle S, S\rangle$, respectively.

It is a simple  and well known  fact   that  a  strongly regular graph  whose valency  is equal to the multiplicity of a non-principal eigenvalue   is either a conference graph, that is an ${\mathsf{SRG}}(n, (n-1)/2, (n-5)/4, (n-1)/4)$,  or an
\begin{eqnarray}\label{g=k}
{\mathsf{SRG}}((n^2+3n-\lambda)^2, n(n^2+3n-\lambda+1), \lambda, n(n+1)),
\end{eqnarray}
for some integer $n$; depending  on $f=g$ or not.
Let $G$ be a graph of the family given by  (\ref{g=k}).
The eigenvalues  of $G$ are  $n$ with the multiplicity $\nu-1-k$ and $\lambda-n^2-2n$ with the multiplicity  $k$. Traditionally,
if $n>0$, then $g=k$ and  $G$ is called a {\sl negative Latin square} graph  and if $n<0$, then $f=k$ and $G$ is called a {\sl pseudo Latin square} graph. Note that if $n<0$, then  $\lambda-n^2-2n>0$ and so $n>-1-\sqrt{1+\lambda}$. This means that,  for a f{}ixed parameter    $\lambda$, there are only  f{}initely many  strongly regular graphs with $f=k$. In this article, we only   deal with strongly regular graphs with $f\neq g$ and  $g=k$.

Let  $G$ be a diamond-free  ${\mathsf{SRG}}(\nu, k, \lambda, \mu)$ in  the family    (\ref{g=k}) with  $0\leqslant\lambda\leqslant n-1$. Fix a  vertex $u\in V(G)$ and  assume that $\langle N(u)\rangle=sK_{\lambda+1}$,  where $s=k/(\lambda+1)$. Letting $H=\langle N[u]\rangle$, we may write
\begin{equation} \label{xy}
\EuScript{A}_G=\left[
\begin{array}{cc}
X & Y \\
Y^\top  & \EuScript{A}_H$$ \\
\end{array}
\right],
\end{equation}
for some matrices $X$ and  $Y$.
Since $\lambda\leqslant n-1$, $n$ is not an eigenvalue of $H$.
With an easy calculation, we f{}ind that
\begin{equation} \label{inv}
n(n+1)^2(n-\lambda)(nI_{k+1}-\EuScript{A}_H)^{-1}=\left[
\begin{array}{c|c}
(aI_{\lambda+1}+\mu J_{\lambda+1})\otimes I_s & b\bmi{j}_k \\\hline
b\bmi{j}_k^\top  & c \\
\end{array}
\right]-J_{k+1},
\end{equation}
where
$a=\mu(n-\lambda)$,  $b=\lambda+1-n$,  $c=(\lambda+1-n)(n+1-\lambda)$, and $\bmi{j}_k$  is    the all one column vector of  length $k$.
For every two vertices $v, w\in\overline{N}(u)$, let $p_u(v, w)=|N(u, v, w)|$ and  $q_u(v, w)$ be the number of pairs  $x\thicksim y$ with  $x\in N(u, v)$ and $y\in N(u, w)$.
Since $g=k$, we   have   $\text{\sl rank}\,(nI_\nu-\EuScript{A}_G)=\text{\sl rank}\,(nI_{k+1}-\EuScript{A}_H)$, which implies by (\ref{xy}) that
\begin{equation}\label{star}
nI_\nu-X=Y(nI_{k+1}-\EuScript{A}_H)^{-1}Y^\top.
\end{equation}
Using (\ref{inv}) and  (\ref{star}),  it is not  hard to see that
\begin{eqnarray}\label{pq}
(n-\lambda+1)p_u(v, w)+q_u(v, w)=\left\{\begin{array}{ll}
   \lambda(n+1), & \text{if $v\thicksim w$}; \\
    \mu, & \text{otherwise},
\end{array}\right.
\end{eqnarray}
for every two vertices $v, w\in\overline{N}(u)$.

Now,  f{}ix a  vertex  $v\in\overline{N}(u)$ and set $t=\lfloor\mu/(n-\lambda+1)\rfloor$.
For $i=0, 1, \ldots, t$, let $M_i(u, v)$ be the set  of all vertices  $x\not\in N[u]\cup N[v]$ with $p_u(v, x)=i$, and put  $m_i(u, v)=|M_i(u, v)|$. By a double counting argument, it is straightforward to f{}ind that
\begin{eqnarray}\label{sys}
\left\{\begin{array}{l}
\displaystyle{\sum_{i=0}^tm_i(u, v)=\nu-2k+\mu-2}; \\
\vspace{-5mm}\\
\displaystyle{\sum_{i=0}^tim_i(u, v)=\mu(k-2\lambda-2)}; \\
\vspace{-4mm}\\
\displaystyle{\sum_{i=0}^t{{i} \choose {2}}m_i(u, v)=(\mu-2){{\mu} \choose {2}} }.
\end{array}\right.
\end{eqnarray}
Notice  that $G$ satisf{}ies (\ref{con}) if and only if $m_0(u, v)\neq0$ for every two   non-adjacent vertices   $u,  v\in V(G)$.

\section*{III. The Proof of Theorem \ref{main}}

In this section, we give a proof of   Theorem \ref{main}. Let $\mathbb{G}$ be a diamond-free   ${\mathsf{SRG}}((n^2+3n-2)^2, n(n^2+3n-1), 2, n(n+1))$,  for some integer $n\geqslant3$, satisfying  (\ref{con}). We will demonstrate  that either  $n=3$ or $n=10$. In the following lemma, we solve the system (\ref{sys})  for  each pair   $u\nsim v$ of vertices of   $\mathbb{G}$.
For any   vertex $u\in V(\mathbb{G})$, we  denote by $\mathnormal{\Phi}(u)$ the partition  of $N(u)$    into  cliques of size $3$.

\begin{lem}\label{sol}
For every  two  non-adjacent vertices   $u,  v\in V(\mathbb{G})$,
the system (\ref{sys}) has the unique  solution
\begin{eqnarray}\label{uniq}
\left\{\begin{array}{l}
m_0(u, v)=2; \\
m_1(u, v)=\cdots=m_{n-1}(u, v)=0; \\
m_n(u, v)=n(n+2)(n^2-1);\\
m_{n+1}(u, v)=2n(n^2-4);\\
m_{n+2}(u, v)=n(n+1).
\end{array}\right.
\end{eqnarray}
Moreover, if $M_0(u, v)=\{a, b\}$, for some vertices $a, b\in V(\mathbb{G})$, then $a\nsim b$,    $p_u(a, b)=0$, and  any  element of $\mathnormal{\Phi}(u)$ which meets $N(v)$, also meets both  $N(a)$ and $N(b)$.
\end{lem}

\begin{proof}{
Fix two non-adjacent vertices   $u,  v\in V(\mathbb{G})$. Since $\mathbb{G}$ satisf{}ies (\ref{con}), there exists a vertex  $a\in M_0(u, v)$.
We f{}irst establish the following steps.

\noindent{\bf{\textsf{Step 1.}}} $\langle M_0(u, v), M_{n+2}(u, v)\rangle$ is complete bipartite.

By contrary, suppose  that $x\in M_0(u, v)$ is not adjacent to  $y\in M_{n+2}(u, v)$.  Since $q_u(v, x)=\mu$,  $p_u(v, y)=2$, and $q_u(v, y)=n+2$, one can easily  deduce  that $q_u(x, y)\geqslant n+2$ and $p_u(x, y)+q_u(x, y)=n+4$. Further,
we have from  (\ref{pq}) that  $(n-1)p_u(x, y)+q_u(x, y)=\mu$. These two   equalities  yield that    $q_u(x, y)=2$,  a contradiction.

\noindent{\bf{\textsf{Step 2.}}} $\langle N(u, a), N(v, a)\rangle$ is $1$-regular.

Consider an  arbitrary vertex  $x\in N(v, a)$.  Since $\langle N[v]\rangle$ is a disjoint union of triangles, $p_u(v, x)=1$ and so
(\ref{pq}) implies that   $q_u(v, x)=n+3$. This shows that $p_u(a, x)+q_u(a, x)=n+4$. Again,   (\ref{pq}) yields that   $p_u(a, x)=1$, as required.

\noindent{\bf{\textsf{Step 3.}}} $m_{n+2}(u, v)\leqslant\mu$.

Consider an  arbitrary vertex   $x\in M_{n+2}(u, v)$. Since $q_u(v, a)=\mu$,   $p_u(v, x)=n+2$, and $q_u(v, x)=2$, we conclude  that  $p_u(a, x)+q_u(a, x)=n+4$.
By Step 1 and  (\ref{pq}), we f{}ind that $p_u(a, x)=1$ and similarly, $p_v(a, x)=1$. Let $N(u, a, x)=\{u'\}$ and $N(v, a, x)=\{v'\}$. Since $\mathbb{G}$ is diamond-free,  $u'\thicksim v'$. It follows from Step 2 that  $m_{n+2}(u, v)\leqslant\mu$, as desired.

\noindent{\bf{\textsf{Step 4.}}} $m_0(u, v)\leqslant2$ and the `Moreover' statement   holds.

For every two vertices $x, y\in M_0(u, v)$, we have   $p_u(x, y)+q_u(x, y)=\mu$ and   by (\ref{pq}),  $(n-1)p_u(x, y)+q_u(x, y)=\epsilon(n+1)$, where  $\epsilon\in\{2, n\}$. This yields that $p_u(x, y)=0$ and $x\nsim y$. Since $\langle N[u]\rangle$ is a disjoint union of triangles, we   must have $m_0(u, v)\leqslant2$. If $M_0(u, v)=\{a, b\}$, then (\ref{pq}) forces that   $q_u(v,a)=q_u(v, b)=\mu$. This  shows clearly  that the `Moreover' statement  is valid.

\noindent{\bf{\textsf{Step 5.}}}   Let $\{u, v_1, w_1\}$ be an  independent set with  $p_u(v_1, w_1)\neq0$.  Then $p_u(v_1, w_1)\geqslant n$.

Let $v_2 \in M_0(u,v_1)$ and $w_2 \in M_0(u,w_1)$. Since $p_u(v_1,w_1)\neq0$, Step 4 shows that  $v_2\neq w_2$. Let $t$ denote  the number of elements in $\mathnormal{\Phi}(u)$ meeting both $N(v_1)$ and $N(w_1)$. Using Step 4 and  (\ref{pq}), we have
\begin{eqnarray}\label{4}
(n-1)p_u(v_i, w_j)+\big(t-p_u(v_i, w_j)\big)=\epsilon_{ij}(n+1), \quad \text{ for } i, j\in\{1, 2\},
\end{eqnarray}
where  $\epsilon_{ij}=2$,  if $v_i\thicksim w_j$ and   $\epsilon_{ij}=n$, otherwise.
Since $n\geqslant3$ and $p_u(v_1, w_1)+p_u(v_1, w_2)+p_u(v_2, w_1)+p_u(v_2, w_2)=t$,
summing up the    four  formulae given in (\ref{4}), we obtain  that $t\leqslant4\mu/(n+2)$.  The    equality  (\ref{4}) for $i=j=1$ yields that  $p_u(v_1, w_1)\geqslant\mu/(n+2)>n-1$, as we wanted to prove.

We are now prepared to solve the system (\ref{sys}) for $\mathbb{G}$.
Obviously, Step 5 means that $m_1(u, v)=\cdots=m_{n-1}(u, v)=0$.
Solving the system (\ref{sys}) in terms of $m_n(u, v),$ $m_{n+1}(u, v),$  $m_{n+2}(u, v)$, we obtain that
\begin{align}
&m_{n}(u, v)=(n+1)(n+2)(n^2-n+1)-{n+2\choose 2}m_0(u, v); \label{tm0}\\
&m_{n+1}(u, v)=2n(n+2)(n-3)+n(n+2)m_0(u, v); \label{tm1}\\
&m_{n+2}(u, v)=2n(n+1)-{n+1\choose 2}m_0(u, v). \label{tm2}
\end{align}
\vspace{-3.9cm}
\begin{eqnarray*}
\left\{\begin{array}{l}
${}$\hspace{10cm}${}$
\vspace{2.2cm}
${}$\end{array}\right.
\end{eqnarray*}
From (\ref{tm2}) and  using Steps 3 and  4,  we deduce  that $m_0(u, v)=2$ and $m_{n+2}(u, v)=n(n+1)$. Now, the solution  (\ref{uniq}) is clearly obtained  from  (\ref{tm0}) and (\ref{tm1}).
}\end{proof}

Consider a vertex  $u\in V(\mathbb{G})$. Obviously,   Lemma \ref{sol}   shows  that  $\overline{N}(u)$ has a  partition  $\mathnormal{\Psi}(u)$  into   independent sets of size $3$   such that  $p_u(x, y)=0$, for every two distinct vertices $x$ and $y$ belonging to an element of   $\mathnormal{\Psi}(u)$. Notice  that for every subsets   $\phi\in\mathnormal{\Phi}(u)$ and $\psi\in\mathnormal{\Psi}(u)$,    $\langle\phi, \psi\rangle$ is either edgeless or  $1$-regular. In the latter case,  we say that $\phi$ and    $\psi$ are {\sl matched} together.

\begin{lem}\label{cicj}
Let $u\in V(\mathbb{G})$ and let $\psi, \psi'$ be two distinct elements of $\mathnormal{\Psi}(u)$. Then   $\langle\psi, \psi'\rangle$ is $r$-regular with $r\in\{0, 1, 2\}$. Moreover, for every two  vertices $v\in\psi$ and  $w\in\psi'$,
\begin{eqnarray*}
p_u(v, w)=\left\{\begin{array}{ll}
   \max\{0, r-1\}, & \text{if $v\thicksim w$}; \\
    n+r, & \text{otherwise}.
\end{array}\right.
\end{eqnarray*}
\end{lem}

\begin{proof}{
Let $v\in\psi$, $\psi'=\{w_1, w_2, w_3\}$, and $t=p_u(v, w_1)+p_u(v, w_2)+p_u(v, w_3)$. By  Lemma \ref{sol},  $t$ is independent of the choice of $v$ in $\psi$ and  $q_u(v, w_i)=t-p_u(v, w_i)$,  for $i=1, 2, 3$. Applying (\ref{pq}), we f{}ind for each  $i$ that  $(n-2)p_u(v, w_i)=\epsilon_i(n+1)-t$, where $\epsilon_i=2$,  if $v\thicksim w_i$ and   $\epsilon_i=n$, otherwise. Summing up these three formulae, we obtain that $\epsilon_1+\epsilon_2+\epsilon_3=t$. It follows from  $n\geqslant3$ that the degrees of the   elements  in $\psi$ as some  vertices  of  $\langle\psi, \psi'\rangle$ are the same. Clearly, a similar property holds  for the  elements of $\psi'$. This shows  that $\langle\psi, \psi'\rangle$  is $r$-regular, for some $r$.  By Lemma \ref{sol}, $m_2(u, v)=0$ and so $r\in \{0, 1, 2\}$. The rest of the proof is straightforward.
}\end{proof}

\begin{lem}\label{cij2}
Let $u\in V(\mathbb{G})$ and let  $\psi=\{v_1, v_2, v_3\}$,   $\psi'=\{w_1, w_2, w_3\}$ be two distinct elements of $\mathnormal{\Psi}(u)$ in which    $\langle\psi, \psi'\rangle$ is  $2$-regular and     $v_i\nsim w_i$,  for $i=1, 2, 3$.  Then
for any element $\{a_1, a_2, a_3\}\in\mathnormal{\Phi}(u)$ matched to  both  $\psi$ and $\psi'$, there is an permutation $\pi\in\langle(1 \, 2 \, 3)\rangle$ such that   $a_i\thicksim v_i$  and $a_i\thicksim w_{\pi(i)}$,  for $i=1, 2, 3$.
\end{lem}

\begin{proof}{
By the contrary and with no loss of generality, suppose that there is an element  $\{a_1, a_2, a_3\}\in\mathnormal{\Phi}(u)$ with  $a_1\in N(v_1, w_1)$, $a_2\in N(v_2, w_3)$, and  $a_3\in N(v_3, w_2)$. Since the neighborhood of each vertex of $\mathbb{G}$ is a disjoint union of triangles,  there is a vertex $x\in N(a_2, v_2, w_3)$. Since $\{a_2, w_3, x\}\in\mathnormal{\Phi}(v_2)$, $\{u, v_1, v_3\}\in\mathnormal{\Psi}(v_2)$, $a_2\thicksim u$,    and $w_3\thicksim v_1$, we deduce  that $x\thicksim v_3$.
Also, since $\{a_2, v_2, x\}\in\mathnormal{\Phi}(w_3)$, $\{u, w_1, w_2\}\in\mathnormal{\Psi}(w_3)$, $a_2\thicksim u$,    and $v_2\thicksim w_1$, we conclude   that $x\thicksim w_2$. Thus  $\langle\{a_3, v_3, w_2, x\}\rangle$  contains a diamond as a subgraph, which forces that  $x\thicksim a_3$. However, this is impossible, since  $\{u, a_1, x\}\subseteq N(a_2, a_3)$.
}\end{proof}

\begin{lem}\label{2cliq}
Let $u\in V(\mathbb{G})$ and let  $\phi, \phi'$ be two distinct elements of $\mathnormal{\Phi}(u)$. Then there is a suitable labeling $\phi=\{a_1, a_2, a_3\}$ and $\phi'=\{b_1, b_2, b_3\}$ such that for any element  $\{v_1, v_2, v_3\}\in\mathnormal{\Psi}(u)$ matched to both  $\phi$ and $\phi'$, the relations  $a_i\thicksim v_i$  and $b_i\thicksim v_{\pi(i)}$ hold, for any  $i\in\{1, 2, 3\}$ and some permutation  $\pi\in\langle(1 \, 2 \, 3)\rangle$.
\end{lem}

\begin{proof}{
Let $\EuScript{R}_{ij\ell}=\{\{v_1, v_2, v_3\}\in\mathnormal{\Psi}(u) \, | \, v_1\in N(a_1, b_i),  v_2\in N(a_2, b_j), v_3\in N(a_3, b_\ell)\}$,  for every $i, j, \ell$ with  $\{i, j, \ell\}=\{1, 2, 3\}.$ Since each pair $a_i\nsim b_j$  has $\mu-1$ common neighbors except  $u$, it is easily seen that  $|\EuScript{R}_{123}|=|\EuScript{R}_{231}|=|\EuScript{R}_{312}|$ and $|\EuScript{R}_{132}|=|\EuScript{R}_{321}|=|\EuScript{R}_{213}|=\mu-1-|\EuScript{R}_{123}|$. Let $\mathcal{S}=\EuScript{R}_{123}\cup\EuScript{R}_{231}\cup\EuScript{R}_{312}$ and $\mathcal{T}=\EuScript{R}_{132}\cup\EuScript{R}_{321}\cup\EuScript{R}_{213}$.  The  assertion of the lemma is  equivalent to that either    $\mathcal{S}=\varnothing$ or $\mathcal{T}=\varnothing$. By  contrary, suppose  that both $\mathcal{S}$ and  $\mathcal{T}$ are not empty.
We show  that the degree of  each vertex of $\langle\mathcal{S}\rangle$ is at  least $2n$. With no loss of generality, consider  $x\in\mathcal{S}\cap N(a_1, b_1)$. It is easily checked by Lemmas \ref{cicj} and \ref{cij2} that $\langle\mathcal{S}, \mathcal{T}\rangle$ is edgeless. Since $b_2\thicksim b_3$, at least one  set in  each of  pairs   $\{N(x, a_2, b_2), N(x, a_2, b_3)\}$ and $\{N(x, a_3, b_2), N(x, a_3, b_3)\}$ is  not empty.
On the other hand, it follows from  (\ref{uniq}) that  either  $p_x(a_i, b_j)=0$ or $p_x(a_i, b_j)\geqslant n$, for every indices  $i, j\in\{2, 3\}$.
This clearly means that the degree  of $x$ as a vertex of $\langle\mathcal{S}\rangle$ is at  least $2n$, as desired. Obviously, the similar property holds for $\langle\mathcal{T}\rangle$. So,  the second largest eigenvalue of  $\langle\mathcal{S}, \mathcal{T}\rangle=\langle\mathcal{S}\rangle\cup\langle\mathcal{T}\rangle$ would be at least $2n$. This  is a contradiction by the  interlacing theorem,  since the second largest eigenvalue of $\mathbb{G}$ is $r=n$.
}\end{proof}

We now proceed to def{}ine  a permutation $\sigma_u$ on $V(\mathbb{G})$ of order $3$ and then demonstrate  that $\sigma_u$ is in fact an automorphism of $\mathbb{G}$. Put  $\sigma_u(u)=u$. Fix an element $\zeta=\{z_1, z_2, z_3\}$ of  $\mathnormal{\Phi}(u)$ and def{}ine  $\sigma_u(z_1)=z_2$, $\sigma_u(z_2)=z_3$,  and $\sigma_u(z_3)=z_1$. We repeatedly do the following process  until  $\sigma_u$ is def{}ined on the whole  $V(\mathbb{G})$:
\vspace{-2mm}
\begin{quotation}
\noindent Assume  that  $\{a_1, a_2, a_3\}\in\mathnormal{\Phi}(u)$ and   $\{v_1, v_2, v_3\}\in\mathnormal{\Psi}(u)$ form a  matched pair with  $a_i\thicksim v_i$, for $i=1, 2, 3$.  If  $\sigma_u$ is already def{}ined on only  one of the  two  triples, then we  def{}ine  $\sigma_u$ on  the other one such that $\sigma_u$ induces the same permutation on indices of elements of  the  two triples.
\end{quotation}
\vspace{-2mm}
Note that we may f{}irst  def{}ine $\sigma_u$ on   the all  elements of $\mathnormal{\Psi}(u)$  matched  with  $\zeta$   and then we can proceed to def{}ine $\sigma_u$  on each element of $\mathnormal{\Phi}(u)$,  since $\mu>1$. Finally, $\sigma_u$ is def{}ined on each element of  $\mathnormal{\Psi}(u)$. We show that $\sigma_u$ is a well def{}ined permutation. For this, it suf{}f{}ices to demonstrate   that
\vspace{-1mm}
\begin{itemize}
\item[$\mathsf{(i)}$] if $\sigma_u$ is def{}ined on two elements  $\psi=\{v_1, v_2, v_3\}$ and  $\psi'=\{w_1, w_2, w_3\}$ in $\mathnormal{\Psi}(u)$ and $\phi=\{a_1, a_2, a_3\}\in\mathnormal{\Phi}(u)$ is matched to  $\psi$ and $\psi'$, then the def{}initions of $\sigma_u$ forced  by $\psi$ and $\psi'$ on $\phi$ are  the same;
\item[$\mathsf{(ii)}$]  if $\sigma_u$ is def{}ined on two elements  $\phi=\{a_1, a_2, a_3\}$ and   $\phi'=\{b_1, b_2, b_3\}$ of  $\mathnormal{\Phi}(u)$ and $\psi=\{v_1, v_2, v_3\}\in\mathnormal{\Psi}(u)$ is matched to   $\phi$ and $\phi'$, then the def{}initions of $\sigma_u$ forced by   $\phi$ and $\phi'$ on $\psi$ are  the same.
\end{itemize}
\vspace{-1mm}
The  assertions $\mathsf{(i)}$  and $\mathsf{(ii)}$ are direct consequences of Lemmas \ref{cij2} and \ref{2cliq}, respectively. For $\mathsf{(i)}$, note that we may assume that  $\zeta$ is matched to $\psi$ and $\psi'$. For $\mathsf{(ii)}$, note that $z_1\in M_i(a_1, b_1)$, for some $i\geqslant1$, and so there is
a vertex $w\in N(z_1, a_1, b_1)$. This shows that there is an element in $\mathnormal{\Psi}(u)$ containing $w$ which  matches to $\zeta$, $\psi$, and $\psi'$.

The above discussion  implies that  $\sigma_u$ is  well def{}ined. Also, from the def{}inition of $\sigma_u$, we easily see that   the subgraphs $\langle N[u]\rangle$ and $\langle N[u], \overline{N}(u)\rangle$ are  f{}ixed by $\sigma_u$. Therefore, applying  (\ref{star}), $\langle\overline{N}(u)\rangle$ is f{}ixed by $\sigma_u$  and hence   $\sigma_u$ is   an automorphism of $\mathbb{G}$.

As we saw in the above, for  each vertex $u\in V(\mathbb{G})$, we can associate   to $u$  two automorphisms of $\mathbb{G}$ of  order $3$,  that  are the inverse of each other. Fix a vertex  $\mathpzc{z}\in V(\mathbb{G})$ and also   f{}ix $\sigma_\mathpzc{z}$ to be one  of the two  automorphisms associated   to  $\mathpzc{z}$. Now, for  any arbitrary  vertex $u\in V(\mathbb{G})$, let  $\sigma_u$ be that automorphism associated   to $u$
satisfying   $\sigma_u(\mathpzc{z})=\sigma^{-1}_\mathpzc{z}(u)$.

\begin{lem}\label{invers}
For every two   vertices  $u, v\in V(\mathbb{G})$, $\sigma_u(v)=\sigma^{-1}_v(u)$.
\end{lem}

\begin{proof}{
In order to prove the lemma, we need to establish a more general result. For any  vertex $u\in V(\mathbb{G})$, f{}ix $^u\tau_u$ to be one of the two automorphisms which  perviously def{}ined at $u$. Also, for each other   vertex $v\in V(\mathbb{G})$, let  $^u\tau_v$ be  that  automorphism def{}ined at $v$ satisfying  $^u\tau_v(u)={^u}\tau^{-1}_u(v)$. Consequently, we  have    $^u\tau^{-1}_v(u)={^u}\tau_u(v)$, for every   vertices $u, v\in V(\mathbb{G})$. We  claim  that  $^a\tau_b(c)={^a}\tau^{-1}_c(b)$,  for every    vertices $a, b, c\in V(\mathbb{G})$.
This   clearly implies  the assertion of the lemma, if we consider $\mathpzc{z}$ instead of $a$.
We will just   prove the claim when  $a, b, c$ are mutually distinct,  since otherwise the claim  follows from the def{}inition.  We consider the following seven  cases.

\noindent{\bf{\textsf{Case 1.}}}    $a\thicksim b$, $a\thicksim c$,  $b\thicksim c$.

In this case,  the claim  is easily  checked  from the def{}inition.

\noindent{\bf{\textsf{Case 2.}}}    $a\thicksim b$, $a\thicksim c$, $b\nsim c$.

Let $\{b, u, u'\}, \{c, v, v'\}\in\mathnormal{\Phi}(a)$ and   $M_0(b, c)=\{w, w'\}$. From $N(b, c, w)=N(b, c, w')=\varnothing$, we f{}ind  that $a\nsim w$ and $a\nsim w'$. Also, from  $a\not\in M_0(w, w')=\{b, c\}$,  one concludes   that $M_0(a, w)$ and $M_0(a, w')$ are disjoint. Let $M_0(a, w)=\{x, x'\}$ and  $M_0(a, w')=\{y, y'\}$. Since  $\{a, u, u'\}\in\mathnormal{\Phi}(b)$, $\{c, w, w'\}\in\mathnormal{\Psi}(b)$, and  $a\thicksim c$, we may, with no loss of generality, assume   that    $u\thicksim w$ and $u'\thicksim w'$. Similarly,  let    $v\thicksim w$ and $v'\thicksim w'$. Without  loss of generality, assume that $^a\tau_a(b)=u$ and $b\thicksim x$. Then  $^a\tau_b(a)={^a}\tau^{-1}_a(b)=u'$, which yields that $^a\tau_b(c)=w'$.  Consider two elements  $\{a, x, x'\}, \{b, c, w'\}\in\mathnormal{\Psi}(w)$. Since $a\in N(b, c)$, Lemma \ref{cicj} yields  that   $\langle\{a, x, x'\}, \{b, c, w'\}\rangle$ is $2$-regular and so we conclude  from   $b\thicksim x$  that     $c\thicksim x'$. Therefore,  $x\thicksim v'$. Since  $^a\tau_a(b)$ has   cycle $(b \, u \,  u')$, it also has cycles $(x \, w \,  x')$ and $(v' \, v \,  c)$. Hence $^a\tau_a(c)=v'$,  which in turn  implies that  $^a\tau^{-1}_c(a)={^a}\tau_a(c)=v'$. So $^a\tau_c$ has cycle $(v' \, a \,  v)$ and so  it also has cycle $(w' \, b \,  w)$. Thus $^a\tau^{-1}_c(b)=w',$ as desired.

\noindent{\bf{\textsf{Case 3.}}}    $a\thicksim b$, $a\nsim c$, $b\thicksim c$.

By the def{}inition,   either  $^b\tau_a={^a}\tau_a$ or $^b\tau_a={^a}\tau^{-1}_a$. We only consider the f{}irst equality. The argument is similar, if the second equality  occurs. We have  $^a\tau_b(a)={^a}\tau^{-1}_a(b)={^b}\tau^{-1}_a(b)={^b}\tau_b(a)$. Since $^a\tau_b$ and $^b\tau_b$ are coincide on $\{a, b\}$, we conclude from the def{}inition  that  $^a\tau_b={^b}\tau_b$.    Also, Case 2 implies that  $^b\tau_c(a)={^b}\tau^{-1}_a(c)={^a}\tau^{-1}_a(c)={^a}\tau_c(a)$, which yields that   $^b\tau_c={^a}\tau_c$.  Therefore,
$^a\tau_b(c)={^b}\tau_b(c)={^b}\tau^{-1}_c(b)={^a}\tau^{-1}_c(b)$,  as required.

\noindent{\bf{\textsf{Case 4.}}}    $N(a, b, c)\neq\varnothing$.

Consider  a vertex $x\in N(a, b, c)$. We assume that    $^x\tau_a={^a}\tau_a$. The argument is similar when $^x\tau_a={^a}\tau^{-1}_a$. Using Cases  1 and 2,  we can write $^x\tau_b(a)={^x}\tau^{-1}_a(b)={^a}\tau^{-1}_a(b)={^a}\tau_b(a)$.  Hence $^x\tau_b={^a}\tau_b$,  and similarly,  $^x\tau_c={^a}\tau_c$.  Therefore, by Cases  1 and 2, we f{}ind that  $^a\tau_b(c)={^x}\tau_b(c)={^x}\tau^{-1}_c(b)={^a}\tau^{-1}_c(b)$, as we wanted to prove.

\noindent{\bf{\textsf{Case 5.}}}    $a\nsim b$, $a\nsim c$,  $b\nsim c$.

If $a\in M_0(b, c)$, then the claim  is easily  checked  from the def{}inition.  So, let $a\not\in M_0(b, c)$, which  means that  there exists  a vertex $x\in N(a, b, c)$. Now we are done by Case  4.

\noindent{\bf{\textsf{Case 6.}}}    $a\nsim b$, $a\nsim c$, $b\thicksim c$.

It   suf{}f{}ices by  Case 4 to assume that $N(a, b, c)=\varnothing$. Let $y, y'\in N(a, b)$ and $z\in N(b, y')$.
Since  $a\nsim b$, we have $y\nsim y'$.
We assume that    $^a\tau_y={^y}\tau_y$.  The argument is similar when $^a\tau_y={^y}\tau^{-1}_y$.
By Case 3, we obtain that $^a\tau_b(y)={^a}\tau^{-1}_y(b)={^y}\tau^{-1}_y(b)={^y}\tau_b(y)$, which yields that $^a\tau_b={^y}\tau_b$.
Since $\langle N(b)\rangle$ and $\langle N(y')\rangle$ are  disjoint unions  of triangles,  $z\not\in N(a)\cup N(c)\cup N(y)$.  It follows from $y'\in N(a, b, z)$ and   Cases 3 and   4  that  $^y\tau_z(b)={^y}\tau^{-1}_b(z)={^a}\tau^{-1}_b(z)={^a}\tau_z(b)$ and thus   $^y\tau_z={^a}\tau_z$.  Moreover,  it follows from $b\in N(c, y, z)$ and    Cases 4 and   5  that  $^y\tau_c(z)={^y}\tau^{-1}_z(c)={^a}\tau^{-1}_z(c)={^a}\tau_c(z)$ and hence     $^y\tau_c={^a}\tau_c$. Since $N(a, b, c)=\varnothing$, we have   $c\nsim y$,  which together Case 3 imply  that $^a\tau_b(c)={^y}\tau_b(c)={^y}\tau_c^{-1}(b)={^a}\tau^{-1}_c(b)$, as desired.

\noindent{\bf{\textsf{Case 7.}}}    $a\thicksim b$, $a\nsim c$, $b\nsim c$.

We assume that    $^c\tau_a={^a}\tau_a$.  The argument for the  case  $^c\tau_a={^a}\tau^{-1}_a$ is similar. We have
$^a\tau_c(a)={^a}\tau^{-1}_a(c)={^c}\tau^{-1}_a(c)={^c}\tau_c(a)$,  which  implies that   $^a\tau_c={^c}\tau_c$. Using Case 6,
$^c\tau_b(a)={^c}\tau^{-1}_a(b)={^a}\tau^{-1}_a(b)={^a}\tau_b(a)$ and so  $^c\tau_b={^a}\tau_b$. Now, we f{}ind that
$^a\tau_b(c)={^c}\tau_b(c)={^c}\tau^{-1}_c(b)={^a}\tau^{-1}_c(b)$, as required.

The proof of the claim is now completed  and  so   the assertion of the lemma follows.
}\end{proof}

In order to continue,  we need   the following result.

\begin{thm}  \label{fixsrg}
{\rm \cite[Theorem\,3.2]{beh}}
If $\pi$ is a non-trivial automorphism of an ${\mathsf{SRG}}(\nu, k, \lambda, \mu)$ with  the second largest  eigenvalue $r$, then the  number of f{}ixed points of $\pi$ is at most $$\frac{\nu}{k-r}\max(\lambda, \mu).$$
\end{thm}

\begin{cor}  \label{fixg}
Each  non-trivial automorphism of $\mathbb{G}$ has  at most $\nu/4$ f{}ixed points.
\end{cor}

\begin{lem}\label{order2}
For every  two   vertices $u_1, u_2\in V(\mathbb{G})$,   $(\sigma_{u_1}\sigma^{-1}_{u_2})^2$ is equal to the   identity.
\end{lem}

\begin{proof}{
For four distinct  vertices $a, b, c, d\in V(\mathbb{G})$,
we call the set  $\{a, b, c, d\}$ to be  {\sl related}  if  either  it is a clique or it is an independent set with  $M_0(a, b)=\{c, d\}$.  Note that every  two distinct  vertices  of $\mathbb{G}$ is contained in a unique related set. Let  $U=\{u_1, u_2, u_3, u_4\}$  be a related set and let    $\rho_{ij}=\sigma_{u_i}\sigma^{-1}_{u_j}$, for every $i, j\in\{1, 2, 3, 4\}$. Consider a vertex $x\in V(\mathbb{G})\setminus U$. By   Lemma \ref{invers}, we f{}ind that $\sigma^{-1}_{\sigma^{-1}_{u_i}(x)}(U)=\{\rho_{1i}(x), \rho_{2i}(x), \rho_{3i}(x), \rho_{4i}(x)\}$ and $\sigma_{u_j}\sigma_x(U)=\{\rho_{j1}(x), \rho_{j2}(x), \rho_{j3}(x), \rho_{j4}(x)\}$ are related, for every $i, j\in\{1, 2, 3, 4\}$. Since every two distinct  vertices  of $\mathbb{G}$ is contained in a unique related  set, it is easily seen that $\sigma^{-1}_{\sigma^{-1}_{u_i}(x)}(U)=\sigma_{u_j}\sigma_x(U)$, for every indices $i\neq j$.
It follows that the  eight sets which we associated to $x$ in the above are the same. Denote the common   set  by $\EuScript{H}_x$.  Note that if $y\in\EuScript{H}_x$, then $\EuScript{H}_x=\EuScript{H}_y$. Therefore,   $\mathscr{P}=\{\EuScript{H}_x \, | \, x\in V(\mathbb{G})\setminus U\}$ is clearly a partition of  $V(\mathbb{G})\setminus U$ into related  sets.

Working towards a contradiction, suppose  that $\rho_{12}\neq \rho_{21}$. Consider an arbitrary element $\EuScript{H}_x\in\mathscr{P}$.
Since $\rho^2_{12}$ is a permutation  on  $\EuScript{H}_x$, $|\EuScript{H}_x|=4$, and $\rho_{12}(x)\neq x$, we obviously deduce that   either $\rho^2_{12}$ has no f{}ixed point in   $\EuScript{H}_x$ or   $\rho^2_{12}$ is the identity      on  $\EuScript{H}_x$. Thus, Corollary \ref{fixg} shows that $\rho^2_{12}$ has no f{}ixed point in at least $\frac{3}{16}\nu-1$ elements of $\mathscr{P}$.

Assume that  $\rho^2_{12}$ f{}ixes no  element of $\EuScript{H}_x=\{x, \rho_{12}(x), \rho_{13}(x), \rho_{14}(x)\}$. So   $\rho_{12}(x)\neq \rho_{21}(x)$.  We claim  that one of  $\rho_{12}\rho_{13}$ or $\rho_{12}\rho_{14}$ is the identity  on  $\EuScript{H}_x$. Note that by Lemma \ref{invers}, $\rho_{ij}(x)\neq \rho_{ij'}(x)$ and  $\rho_{ij}(x)\neq \rho_{i'j}(x)$ whenever  $i\neq i'$ and $j\neq j'$. We clearly have $\rho_{21}(x)\in\{\rho_{13}(x), \rho_{14}(x)\}$.  Suppose that  $\rho_{21}(x)=\rho_{13}(x)$. Since the  eight sets which we associated to $x$ in the f{}irst paragraph  of the proof are equal, one concludes that  the elements of   $\EuScript{H}_x\setminus\{x\}$ are
\begin{eqnarray*}
\left\{\begin{array}{ccc}
\rho_{12}(x)=\rho_{24}(x)=\rho_{31}(x), \\  \rho_{13}(x)=\rho_{21}(x)=\rho_{34}(x),\\  \rho_{14}(x)=\rho_{23}(x)=\rho_{32}(x).
\end{array}\right.
\end{eqnarray*}
It is then easy to check that  $\rho_{12}\rho_{13}$  is the identity  on  $\EuScript{H}_x$. With a similar argument, one deduces that if  $\rho_{21}(x)=\rho_{14}(x)$, then  $\rho_{12}\rho_{14}$  is the identity  on  $\EuScript{H}_x$. This establishes the claim.

Note that none of  $\rho_{12}\rho_{13}$ and  $\rho_{12}\rho_{14}$ are trivial. For instance,   if $\rho_{12}\rho_{13}(u_1)=u_1$,
then $\sigma^{-1}_{u_2}\sigma_{u_1}\sigma^{-1}_{u_3}(u_1)=u_1$ and so by Lemma \ref{invers}, we f{}ind  that
$\sigma_{u_2}(u_1)=\sigma_{u_1}\sigma^{-1}_{u_3}(u_1)=\sigma^{-1}_{u_1}(u_3)=\sigma_{u_3}(u_1)$,
which means that  $u_2=u_3$, a contradiction.
Therefore,   one of $\rho_{12}\rho_{13}$ or   $\rho_{12}\rho_{14}$ is a non-trivial  automorphism of $\mathbb{G}$ which is the identity on  at least $\frac{3}{32}\nu-\frac{1}{2}$ elements of $\mathscr{P}$.   It follows from  Corollary \ref{fixg} that $\frac{3}{8}\nu-2\leqslant\frac{1}{4}\nu$, which it contradicts  $n\geqslant3$.
}\end{proof}

\begin{lem}\label{gorder}
The group $\mathnormal{\Gamma}$  generated  by   $\{\sigma_u\sigma^{-1}_v\, |\, u, v\in V(\mathbb{G})\}$ is Abelian and it acts transitively on $V(\mathbb{G})$.
\end{lem}

\begin{proof}{
Consider the  arbitrary vertices $u, v, x, y\in V(\mathbb{G})$.  By  Lemma \ref{invers}, $\sigma_v \sigma^{-1}_{\sigma^{-1}_u(v)}(u)=v$, meaning  that $\mathnormal{\Gamma}$ acts transitively on $V(\mathbb{G})$.
Applying  Lemma \ref{order2}, we  have $(\sigma_u\sigma^{-1}_v)(\sigma_x\sigma^{-1}_y)=\sigma_u\sigma^{-1}_x\sigma_v\sigma^{-1}_y=
\sigma_x\sigma^{-1}_u\sigma_y\sigma^{-1}_v=(\sigma_x\sigma^{-1}_y)(\sigma_u\sigma^{-1}_v)$. So,  $\mathnormal{\Gamma}$ is Abelian.
}\end{proof}

\begin{lem}\label{mthm}
The order of $\mathbb{G}$ is either $256$ or $16384$.
\end{lem}

\begin{proof}{
Applying   Lemmas \ref{order2}  and \ref{gorder}, we f{}ind that $\mathbb{G}$ admits a transitive  automorphism group whose order is a power of $2$. It follows from the orbit-stabilizer theorem  that   $n^2+3n-2=2^t$, for some   integer $t$. We have   $(2n+3)^2=2^{t+2}+17$. Using  a result in    \cite[p.\,401]{beu}, we obtain that  $(n, t)\in\{(1,  1), (2, 3), (3, 4), (10, 7)\}$. Since $n\geqslant 3$, we  conclude that   $(n,\nu)\in\{(3, 256), (10, 16384)\}$.
}\end{proof}

Now, the proof of Theorem \ref{main} is f{}inally completed after proving  Lemma \ref{mthm}. Notice  that we employed   the assumption (\ref{con}) only in the proof of Lemma \ref{sol}. As mentioned before, we believe  that (\ref{con})    automatically holds for  any   diamond-free ${\mathsf{SRG}}((n^2+3n-2)^2, n(n^2+3n-1), 2, n(n+1))$.

\section*{IV. Partial Quadrangle {\sf PQ}(3, 35, 20)}

In the following, we demonstrate  that there exists no  ${\mathsf{PQ}}(3, 35, 20)$, or  equivalently, there is  no diamond-free  ${\mathsf{SRG}}(676, 108, 2, 20)$. Notice  that this strongly regular graph belongs to the family (\ref{g=k}) with $n=4$ and $\lambda=2$.

\begin{thm}\label{n=4}
There exists   no diamond-free  ${\mathsf{SRG}}(676, 108, 2, 20)$.
\end{thm}

\begin{proof}{
Suppose, toward a contradiction, that $G$ is a  diamond-free  ${\mathsf{SRG}}(676, 108, 2, 20)$. Consider two non-adjacent vertices  $u, v\in V(G)$. Since $G$  is diamond-free, there are  vertices $w\in N(u)$ and $v', v''\in\overline{N}(u)$ such that $\{v, v', v'', w\}$ is a clique.  For $i=0, 1, 2, 3$, assume that $s_i$ is the number of  cliques $\mathnormal{\Omega}$ in $N(u)$ of size $3$ such that $\langle\mathnormal{\Omega}, \{v, v', v''\}\rangle$ has $i$ edges. By a double counting argument, we  f{}ind that
\begin{eqnarray*}\label{sub}
\left\{\begin{array}{l}
s_0+s_1+s_2+s_3=35;\\
s_1+2s_2+3s_3=57; \\
s_2+3s_3=21,
\end{array}\right.
\end{eqnarray*}
which gives $s_0=-s_3-1$, a contradiction.
}\end{proof}

\section*{Acknowledgements}

The research of the f{}irst author  was in part supported by a grant from IPM (No. 91050405).

{}
\end{document}